\documentclass{article}

\usepackage{amssymb,amsmath,amsthm,enumitem}

\setlist{nolistsep}

\newtheorem{theorem}{Theorem}[section]
\newtheorem{lemma}[theorem]{Lemma}
\newtheorem{proposition}[theorem]{Proposition}

\theoremstyle{definition}

\newcommand{\iin}{\,\mathord\in\,}
\newcommand{\abs}[1]{\lvert#1\rvert}
\newcommand{\abslarge}[1]{\left\lvert#1\right\rvert}
\newcommand{\norm}[1]{\lVert#1\rVert}
\newcommand{\up}{\mathsf{UP}}
\newcommand{\ueb}{\mathsf{UEB}}
\newcommand{\ue}{\mathsf{UE}}
\newcommand{\UMeas}{\mathsf{{\mathfrak{M}}_u}}
\newcommand{\FMeas}{\mathsf{{\mathfrak{M}}_F}}
\newcommand{\Mol}{\mathsf{Mol}}
\newcommand{\Ub}{\mathsf{U_b}}
\newcommand{\Uall}{\mathsf{U}}
\newcommand{\psm}{\Delta}
\newcommand{\psmtil}{\widetilde{\psm}}
\newcommand{\psmfun}[2]{#1^\bullet\!#2}
\newcommand{\pset}{\mathcal{P}}
\newcommand{\hypers}{\mathsf{H}}
\newcommand{\rightu}{\mathsf{u_R}}
\newcommand{\leftu}{\mathsf{u_L}}
\newcommand{\upu}{\mathsf{u_\vee}}
\newcommand{\lowu}{\mathsf{u_\wedge}}
\newcommand{\genu}[1]{\mathsf{u_{#1}}}
\newcommand{\initL}{\iota_{\vee L}}
\newcommand{\initR}{\iota_{\vee R}}
\newcommand{\finL}{\iota_{L\wedge}}
\newcommand{\finR}{\iota_{R\wedge}}
\newcommand{\initfin}{\iota_{\vee\wedge}}
\newcommand{\extm}[1]{\overline{#1}}
\newcommand{\real}{\mathbb{R}}
\newcommand{\compln}[1]{\widehat{#1}}
\newcommand{\mmeas}{\mathfrak{m}}
\newcommand{\nmeas}{\mathfrak{n}}
\newcommand{\pmass}{\partial}

\title{An embedding theorem for uniform measures \\
       on topological groups}

\author{Jan Pachl       \\
Toronto, Ontario, Canada}

\date{June 5, 2023}

\begin{document}

\maketitle

\begin{abstract}
The embedding theorem of Roelcke and Dierolf for the completions
of four standard uniform structures on topological groups and their quotients
holds more generally for spaces of uniform measures.
The natural mappings between the four spaces of uniform measures
on a topological group are injective and their restrictions
to positive cones are topological embeddings.
The same holds for spaces of uniform measures on the quotient
of a topological group by a neutral subgroup.
\end{abstract}


\section{Summary}
    \label{sec:summary}

In their study of uniform structures on topological groups and their
quotients, Roelcke and Dierolf~\cite{Roelcke1981ust}
identify several situations in which the natural mapping
between completions of two uniform spaces is a topological embedding.
They establish conditions for that to occur, and in particular prove
that for every topological group $G$ the mappings
in the following commutative diagram are topological embeddings.
Here $\leftu G$, $\rightu G$, $\upu G$ and $\lowu G$
are $G$ with its left, right, upper and lower uniform structures.
The identity mappings
$\upu G \to \leftu G$,
$\upu G \to \lowu G$,
$\upu G \to \rightu G$,
$\leftu G \to \lowu G$
and $\rightu G \to \lowu G$,
being uniformly continuous,
uniquely extend to uniformly continuous mappings between completions
of the four uniform spaces.

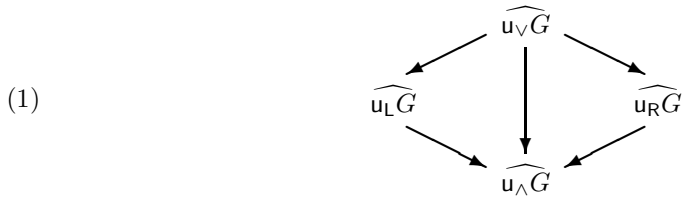
\begin{picture}(200,90)(-80,50)
\thicklines
\put(-90,80){\makebox(30,20){(1)}}

\put(100,110){\makebox(30,20){$\compln{\upu G}$}}
\put(50,80){\makebox(30,20){$\compln{\leftu G}$}}
\put(150,80){\makebox(30,20){$\compln{\rightu G}$}}
\put(100,50){\makebox(30,20){$\compln{\lowu G}$}}

\put(100,115){\vector(-2,-1){30}}
\put(130,115){\vector(2,-1){30}}
\put(160,80){\vector(-2,-1){30}}
\put(70,80){\vector(2,-1){30}}
\put(115,110){\vector(0,-1){40}}
\end{picture}

A special case of the embedding theorem
of Roelcke and Dierolf~\cite[10.9]{Roelcke1981ust}
states that the five mappings in the commutative diagram (1)
are topological embedding,
and that the image of $\compln{\upu G}$ in $\compln{\lowu G}$
is the intersection of the images of $\compln{\leftu G}$
and $\compln{\rightu G}$.

Every uniform space $X$ embeds into a ``linear completion'' of $X$,
the topological vector space $\UMeas(X)$
of so called \emph{uniform measures}~\cite[Ch.6]{Pachl2013usm}.
The completion $\compln{X}$ of $X$ naturally identifies
with a subset of the positive cone $\UMeas(X)^+$ of $\UMeas(X)$,
so that $X\subseteq\compln{X}\subseteq\UMeas(X)^+\subseteq\UMeas(X)$.
Moreover, every uniformly continuous mapping $X\to Y$
uniquely extends to a continuous linear mapping
$\UMeas(X) \to \UMeas(Y)$,
which maps $\compln{X}$ into $\compln{Y}$
and $\UMeas(X)^+$ into $\UMeas(Y)^+$.
Hence the mappings in (1) are restrictions of those in the
following diagram.

\begin{picture}(200,80)(-80,50)
\thicklines
\put(-90,80){\makebox(30,20){(2)}}

\put(100,110){\makebox(30,20){$\UMeas(\upu G)$}}
\put(50,80){\makebox(30,20){$\UMeas(\leftu G)$}}
\put(150,80){\makebox(30,20){$\UMeas(\rightu G)$}}
\put(100,50){\makebox(30,20){$\UMeas(\lowu G)$}}

\put(90,115){\vector(-2,-1){30}}
\put(135,115){\vector(2,-1){30}}
\put(165,80){\vector(-2,-1){30}}
\put(60,80){\vector(2,-1){30}}
\put(115,110){\vector(0,-1){40}}
\end{picture}

The current paper extends the embedding theorem to the larger spaces
in diagram (2).
By Theorem~\ref{th:embedgroup} in section~\ref{sec:embgroup},
the mappings in (2) are injective
and their restrictions to the positive cones $\UMeas(\cdot)^+$
are topological embedding,
and the image of $\UMeas(\upu G)$ in $\UMeas(\lowu G)$
is the intersection of the images of $\UMeas(\leftu G)$
and $\UMeas(\rightu G)$.
This strengthens the embedding theorem for the mappings in (1).

Note that if $\leftu G \neq \rightu G$ then the mappings in (2) themselves,
although continuous and injective, are not topological embeddings.
In fact, a topological embedding of one topological vector space
into another is a uniform embedding for the additive uniform structures,
and each $\genu{\bullet} G$ is a uniform subspace of
$\UMeas(\genu{\bullet} G)$.

Diagram (1) is a special case of a more general situation in which
the quotient space $Q=G/{H}$ takes the place of $G$.

\begin{picture}(200,90)(-40,50)
\thicklines
\put(-90,80){\makebox(100,20){(3)}}

\put(100,110){\makebox(30,20){$\compln{Q_\vee}$}}
\put(50,80){\makebox(30,20){$\compln{\leftu G/{H}}$}}
\put(150,80){\makebox(30,20){$\compln{\rightu G/{H}}$}}
\put(100,50){\makebox(30,20){$\compln{Q_\wedge}$}}

\put(100,115){\vector(-2,-1){30}}
\put(130,115){\vector(2,-1){30}}
\put(160,80){\vector(-2,-1){30}}
\put(70,80){\vector(2,-1){30}}
\put(115,110){\vector(0,-1){40}}

\put(270,110){\makebox(30,20){$Q_\vee=(\leftu G/{H})\vee (\rightu G/{H})$}}
\put(270,95){\makebox(30,20){$Q_\wedge=(\leftu G/{H})\wedge (\rightu G/{H})$}}
\end{picture} \\
When ${H}$ is the trivial one-element subgroup,
this is the special case in diagram (1).
Under the necessary assumption that ${H}$ is a neutral subgroup of $G$,
the Roelcke--Dierolf embedding theorem~\cite[10.9]{Roelcke1981ust}
states that the five mappings in the commutative diagram (3)
are topological embedding,
and that the image of $\compln{Q_\vee}$ in $\compln{Q_\wedge}$
is the intersection of the images of $\compln{\leftu G/{H}}$
and $\compln{\rightu G/{H}}$.

As in the special case above,
the mappings in (3) are restrictions of those between the spaces
$\UMeas(\cdot)$.

\begin{picture}(200,80)(-40,50)
\thicklines
\put(-90,80){\makebox(100,20){(4)}}

\put(100,110){\makebox(30,20){$\UMeas(Q_\vee)$}}
\put(40,80){\makebox(30,20){$\UMeas(\leftu G/{H})$}}
\put(156,80){\makebox(30,20){$\UMeas(\rightu G/{H})$}}
\put(100,50){\makebox(30,20){$\UMeas(Q_\wedge)$}}

\put(90,115){\vector(-2,-1){30}}
\put(135,115){\vector(2,-1){30}}
\put(165,80){\vector(-2,-1){30}}
\put(60,80){\vector(2,-1){30}}
\put(115,110){\vector(0,-1){40}}
\end{picture} \\
Section~\ref{sec:embquot} deals with this more general situation.
Again the result is an embedding theorem that
strengthens the embedding theorem for the mappings in (3).

Embedding theorems in section~\ref{sec:embgroup} and~\ref{sec:embquot}
are instances of a general result in section~\ref{sec:general}.
Roelcke and Dierolf also derive their embedding theorem from a general result,
formulated in terms of uniform entourages and Cauchy filters.
The approach in the current paper is somewhat different,
using properties of continuous pseudometrics.

When $X$ is a complete metric space, $\UMeas(X)$ identifies with
the space of finite Radon measures on $X$~\cite[5.28]{Pachl2013usm}.
Hence for metrizable groups it is easy to derive
the embedding theorem for (2) directly from the embedding theorem for (1),
and the one for (4) from that for (3).
Whether the same can be done for general non-metrizable groups is not clear.
I do not know a direct proof of the embedding theorem for (2)
from that for (1), nor (4) from (3), for general topological groups.

Besides $\UMeas(X)$,
another candidate for the role of ``linear completion''
of a uniform space $X$ is the space $\FMeas(X)$
of \emph{free uniform measures}~\cite[Ch.10]{Pachl2013usm},
i.e. the free complete locally convex space over $X$.
In section~\ref{sec:freeunif} I point out that the embedding theorem
fails for $\FMeas$ in place of $\UMeas$.
When $\FMeas$ replaces $\UMeas$, the mappings in (2) and (4),
although injective, are not necessarily topological embeddings.


\section{Preliminaries}

Throughout the paper,
notation and terminology are as in~\cite{Pachl2013usm},
with some additions noted here.
Spaces and groups are by definition non-empty.
Functions are real-valued,
and linear spaces are over the field $\real$ of reals.

When $T_0$ and $T_1$ are topological spaces,
a mapping $\varphi\colon T_0 \to T_1$ is a \emph{topological embedding} iff
it is a homeomorphism between $T_0$ and the subspace $\varphi(T_0)$
of $T_1$.

\subsection{Uniform spaces and pseudometrics}

For the key results in this paper nothing is lost by assuming
that topological groups and uniform spaces are Hausdorff (a.k.a separated).
However, in dealing with quotient spaces it is hard to avoid
non-Hausdorff structures.
Following Isbell~\cite{Isbell1964us},
I use the term \emph{preuniformity} (or \emph{preuniform structure}) for
what is a \emph{uniformity} in the terminology of~\cite{Roelcke1981ust};
that is, a structure that is not necessarily Hausdorff.
I reserve the term \emph{uniformity} for a Hausdorff preuniformity.
A \emph{preuniform space} is a set $S$ with a preuniformity on $S$,
and a \emph{uniform space} is a set $S$ with a uniformity on $S$.

A preuniformity is defined by its set of uniformly continuous pseudometrics.
When $X$ is a preuniform space,
$\up(X)$ denotes the set of all its uniformly continuous pseudometrics.

For any set $\pset$ of pseudometrics on a set $S$,
the \emph{preuniformity (projectively) induced} by $\pset$
is the coarsest preuniformity on $S$ for which all pseudometrics in $\pset$
are uniformly continuous.
We also say that $S$ with that preuniformity is the
\emph{preuniform space induced} by $\pset$.
It is a uniform space if and only if $\pset$ separates the points of $S$.

For two preuniform spaces $X$ and $Y$ on the same set $S$,
the \emph{supremum} $X\vee Y$ of $X$ and $Y$ is the coarsest
preuniform space on $S$ that is finer than both $X$ and $Y$.
The \emph{infimum} $X\wedge Y$ of $X$ and $Y$ is the finest
preuniform space on $S$ that is coarser than both $X$ and $Y$.

The preuniformity of $X\vee Y$ is induced by the set $\up(X)\cup\up(Y)$
of pseudometrics on $S$,
and $\up(X\wedge Y)=\up(X)\cap\up(Y)$.
A mapping from any preuniform space to $X\vee Y$ is uniformly continuous
if and only if it is uniformly continuous as a mapping to $X$ and to $Y$.
A mapping from $X\wedge Y$ to any preuniform space is uniformly continuous
if and only if it is uniformly continuous as a mapping from $X$ and from $Y$.

For a pseudometric $\psm$ on a set $S$ and $x,z\iin S$,
write $\psmfun{z}{\psm}(x):=\psm(x,z)$.
That defines a real-valued function $\psmfun{z}{\psm}$ on $S$.
If $X$ is a preuniform space and $\psm\iin\up(X)$ then $\psmfun{z}{\psm}$
is uniformly continuous on $X$ because it is $\psm$-Lipschitz.

For a pseudometric $\psm$ on a set $S$, $x\iin S$ and
non-empty subsets $A,B$ of $S$, define
\begin{align*}
\psm(A,B) & := \inf \{ \psm(a,b) \mid a\iin A, b \iin B \} \\
\psm(x,A) & := \psm(A,x) := \psm(A,\{x\})                  \\
\psmtil(A,B)
          & := \max \left( \sup_{a\in A} \psm(a,B), \sup_{b\iin B} \psm(A,b)
                    \right)
\end{align*}
If $\psm$ is bounded then $\psmtil(A,B)$ is finite.
In that case $\psmtil$ is a pseudometric on the set of
non-empty subsets of $S$.

In the notation of~\cite[\S 4.2]{Pachl2013usm},
$\psmtil$ is $\psm^\hypers$.
For a uniform space $X$, pseudometrics $\psmtil$ define
the \emph{hyperspace uniformity}~\cite[Ch.II]{Isbell1964us},
also known as the \emph{Hausdorff uniformity}~\cite[Ch.5]{Roelcke1981ust},
on the set of non-empty closed subsets of $X$.

\subsection{Topological groups}

Topological groups are assumed to be Hausdorff.
For a topological group $G$,
the identity element of $G$ is denoted by $e_G$.
The uniform spaces $\rightu G$, $\leftu G$, $\upu G$ and $\lowu G$
on the set $G$ are defined as follows.
Since $G$ is Hausdorff,
$\rightu G$, $\leftu G$, $\upu G$ and $\lowu G$ are in fact uniform,
not merely preuniform, spaces.
\begin{itemize}
\item
The \emph{right uniformity} of $G$
is induced by the set of all (equivalently, all bounded)
right invariant continuous pseudometrics on $G$.
The uniform space $\rightu G$
($\mathsf{r} G$ in the notation of~\cite{Pachl2013usm})
is the set $G$ with the right uniformity.
\item
The \emph{left uniformity} of $G$
is induced by the set of all (equivalently, all bounded)
left invariant continuous pseudometrics on $G$.
The uniform space $\leftu G$ is the set $G$ with the left uniformity.
\item
The uniformity of $\upu G:=\leftu G \vee \rightu G$
is the \emph{upper uniformity}.
\item
The uniformity of $\lowu G:= \leftu G \wedge \rightu G$ is
the \emph{lower} or \emph{Roelcke uniformity},
\end{itemize}
The upper uniformity is induced by the set of all
(equivalently, all bounded)
continuous pseudometrics on $G$ that are right or left invariant.
Note however that in general the lower uniformity is not induced
by a set of pseudometrics that are both right and left invariant.

A subgroup ${H}$ of a topological group $G$ is \emph{neutral} iff for every
neighbourhood $U$ of $e_G$ there exists a neighbourhood $V$ of $e_G$
such that ${H}V \subseteq U{H}$.

When ${H}$ is a closed subgroup of a topological group $G$,
denote by $G/{H}$ the set of left cosets $x{H}$, $x\iin G$,
and by $q$ the quotient mapping $G\to G/{H}$.
Structures on $G$ yield the corresponding \emph{final structures} on $G/{H}$:
\begin{itemize}
\item
The \emph{quotient topology} is the finest topology on the set $G/{H}$
for which the mapping $q$ is continuous.
\item
When $\alpha$ is one of $R$, $L$, $\vee$ and $\wedge$,
the corresponding \emph{quotient preuniformity} is the finest
preuniformity on the set $G/{H}$ for which the mapping $q$
is uniformly continuous from $\genu{\alpha} G$ to $G/{H}$.
The preuniform space $\genu{\alpha} G/{H}$ is the set $G/{H}$
with the quotient preuniformity.
\end{itemize}
Since $G$ is assumed to be Hausdorff and ${H}$ closed,
$\rightu G/{H}$ and $\upu G/{H}$ are in fact uniform spaces
and their topologies
coincide with the quotient topology of $G/{H}$~\cite[5.21]{Roelcke1981ust}.
However, $\leftu G/{H}$ and $\lowu G/{H}$ need not be Hausdorff;
even when they are,
their topologies may be strictly coarser than the quotient
topology~\cite[Ch.5]{Roelcke1981ust}.
But if ${H}$ is a closed neutral subgroup of $G$ then
$\leftu G/{H}$ and $\lowu G/{H}$ are Hausdorff and their
topologies coincide with the quotient topology~\cite[5.28]{Roelcke1981ust}.

By Proposition~7.7 in~\cite{Roelcke1981ust},
the preuniformity of $\leftu G/{H}$ is induced
by the set of all left invariant continuous pseudometrics on $G/{H}$.

If $\psm$ is a bounded right invariant continuous pseudometric on $G$
then $\psmtil(A,B)=\psm(A,B)$ for $A,B\iin G/{H}$,
and the restriction of the pseudometric $\psmtil$ to $G/{H}$
is continuous in the quotient topology on $G/{H}$.
By Proposition~7.6 in~\cite{Roelcke1981ust},
the uniformity of $\rightu G/{H}$ is induced by the set of all
restrictions of $\psmtil$ to $G/{H}$,
where $\psm$ runs through bounded right invariant continuous
pseudometrics on $G$.

\subsection{Uniform measures}

This subsection is a short review of definitions and results
in~\cite{Pachl2013usm}.

When $H$ is a Banach lattice with positive cone $H^+$,
for each $h\iin H$ the elements $h^+, h^- \iin H^+$ are defined by
$h^+ := h \vee 0$, $h^- := (-h) \vee 0$.
Thus $h=h^+ - h^-$.

When $X$ is a uniform space,
$\Uall(X)$ denotes the space of uniformly continuous
real-valued functions on $X$,
and $\Ub(X)$ denotes the space of bounded functions in $\Uall(X)$.
With the sup norm $\norm{\cdot}$ and the pointwise order,
$\Ub(X)$ is a Banach lattice.
A subset $\mathcal{H}$ of $\Ub(X)$ is said to be $\ueb(X)$
if the functions in $\mathcal{H}$ are uniformly equicontinuous
and uniformly bounded in the sup norm.

The Banach space dual $\Ub(X)^\ast$ of $\Ub(X)$
is also the order-bounded dual of $\Ub(X)$.
It is a Banach lattice in which
\[
\mmeas^+(f) = \sup \;\{ \mmeas(g) \mid g\iin\Ub(X)^+ \;\text{and } g \leq f\}
\quad\text{for } \mmeas\iin\Ub(X)^\ast, f \iin \Ub(X)^+ .
\]
Besides the norm topology,
two other topologies on $\Ub(X)^\ast$ are of interest:
\begin{itemize}
\item
The topology of uniform convergence on $\ueb(X)$ sets,
called the \emph{$\ueb$ topology}.
\item
The topology of simple convergence on the elements of $\Ub(X)$,
the \emph{$\Ub(X)$-weak topology}.
\end{itemize}
The \emph{$\ueb$ uniformity} on $\Ub(X)^\ast$ is the
additive uniformity of the $\ueb$ topology.
Unless stated otherwise, in this paper the space $\Ub(X)^\ast$ and its subsets
are always considered as topological and uniform spaces with their
$\ueb$ topology and $\ueb$ uniformity.

For $x\iin X$, the \emph{point mass at $x$} is the functional
$\pmass(x)\iin\Ub(X)^\ast$ defined by $\pmass(x)(f):=f(x)$ for
$f\iin\Ub(X)$.
The mapping $\pmass\colon X \to \Ub(X)^\ast$ is a uniform isomorphism
between the uniform space $X$ and its image
$\pmass(X)\subseteq\Ub(X)^\ast$ with its $\ueb$ uniformity.

We identify each $x\iin X$ with $\pmass(x)\iin\Ub(X)^\ast$,
and denote by $\compln{X}$ the $\ueb$ closure of $X$ in $\Ub(X)^\ast$.
Then $\compln{X}$ with the $\ueb$ uniformity is a complete
uniform space in which $X$ is a uniform subspace.
Thus $\compln{X}$ is a uniform completion of $X$.

The space of finite linear combinations of point masses in $\Ub(X)^\ast$
(\emph{molecular measures} on~$X$) is denoted by $\Mol(X)$.
The $\ueb$ closure of $\Mol(X)$ in $\Ub(X)^\ast$,
the space of \emph{uniform measures} on~$X$, is denoted by $\UMeas(X)$.
The space $\UMeas(X)$ is $\ueb$ complete,
and it is a band in the Banach lattice $\Ub(X)^\ast$;
in particular $\mmeas^+, \mmeas^- \iin \UMeas(X)$ whenever
$\mmeas\iin\UMeas(X)$.
The elements of $\UMeas(X)$ are characterized by a continuity property
that will be used further on:
A linear functional $\mmeas$ on $\Ub(X)$ is in $\UMeas(X)$ if and only if
$\mmeas$ is continuous on every $\ueb(X)$ subset of $\Ub(X)$
in the topology of pointwise convergence on~$X$.

From the foregoing we have
$X\subseteq\compln{X}\subseteq\UMeas(X)^+ \subseteq\UMeas(X)$.
The assignment of $\UMeas(X)$ to $X$ is functorial:
Every uniformly continuous mapping $\varphi\colon X\to Y$
uniquely extends to a continuous linear mapping
$\UMeas(\varphi)\colon \UMeas(X)\to \UMeas(Y)$,
which maps $\compln{X}$ to $\compln{Y}$
and $\UMeas(X)^+$ to $\UMeas(Y)^+$.


\section{General embedding theorem}
    \label{sec:general}

This section deals with the spaces $\UMeas(X)$, $\UMeas(Y)$,
$\UMeas(X\vee Y)$ and $\UMeas(X\wedge Y)$
for two uniform spaces $X$ and $Y$ on the same set of points.
The sufficient condition in Theorem~\ref{th:embedgen} will be used
in sections~\ref{sec:embgroup} and~\ref{sec:embquot}
to derive embedding theorems for topological groups and quotient spaces.

\begin{lemma}
    \label{lem:suffcond}
Let $\pset$ be a point-separating set of pseudometrics on a set $S$.
Let $X$ be the uniform space induced by $\pset$,
and let $Y$ be another uniform space on $S$ which is coarser than $X$.
Assume that $\psmfun{z}{\psm}\iin\Uall(Y)$ for all $\psm\iin\pset$
and $z\iin S$.
Let $\mmeas\iin\UMeas(X)$, $f\iin\Ub(X)^+$ and $\varepsilon>0$.
Then there exists $g\iin\Ub(Y)^+$ such that $g\leq f$ and
$\abs{\mmeas(f)-\mmeas(g)} < \varepsilon$.
\end{lemma}

\begin{proof}
This is a minor variation of the proof of~\cite[6.21]{Pachl2013usm}.
Let $\varepsilon^\prime >0$ be such that
$\varepsilon = \varepsilon^\prime \norm{\mmeas} + \varepsilon^\prime$.
By~\cite[2.6]{Pachl2013usm}
there are $r>0$ and a pseudometric
$\psm=\psm_0 \vee \psm_1 \vee \dots \vee \psm_k$,
where each $\psm_j$ is in $\pset$,
such that $\abs{f(x)-f(y)} < \max(\varepsilon^\prime,r\psm(x,y))$
for all $x,y\iin S$.

For every finite set $K\subseteq S$ and $y\iin S$ let
$g_K (y) := \max_{x\in K} \; (f(x) - \varepsilon^\prime - r\psm(x,y))^+$.
By the assumption, for every $x\iin S$ the function
$\psmfun{x}{\psm} = \max_j \psmfun{x}{\psm_j}$ is in $\Uall(Y)$.
Thus $g_K\iin\Ub(Y)^+$.
For $x,y\iin S$ we have $f(x)-f(y)<\varepsilon^\prime + r\psm(x,y)$,
hence $f(x)-\varepsilon^\prime-r\psm(x,y)<f(y)$,
hence $g_K(y) \leq f(y)$.

The set $\{g_K\mid \text{finite } K \subseteq X\}$ is $\ueb(X)$ and
the net $\{g_K\}_K$ indexed by the directed set of finite subsets of $X$
converges pointwise upwards to some $f_0\iin \Ub(X)^+$.
Clearly $f - \varepsilon^\prime \leq f_0 \leq f$.
Since $\mmeas\iin\UMeas(X)$,
there is $K$ for which $\abs{\mmeas(f_0) - \mmeas(g_K)}<\varepsilon^\prime$.
It follows that
\[
\abs{\mmeas(f)-\mmeas(g_K)}
\leq \abs{\mmeas(f)-\mmeas(f_0)} + \abs{\mmeas(f_0)-\mmeas(g_K)}
< \varepsilon^\prime \norm{\mmeas} + \varepsilon^\prime
= \varepsilon.   \qedhere
\]
\end{proof}

\begin{theorem}
    \label{th:prep}
Let $\pset$ be a point-separating set of pseudometrics on a set $S$.
Let $X$ be the uniform space induced by $\pset$.
Let $Y$ be another uniform space on $S$ which is coarser than $X$,
so that the identity mapping $\iota\colon X\to Y$ is uniformly continuous.
Assume that $\psmfun{z}{\psm}\iin\Uall(Y)$ for all $\psm\iin\pset$
and $z\iin S$.
\begin{enumerate}
\item
The mapping $\UMeas(\iota)\colon \UMeas(X) \to \UMeas(Y)$
is injective.
\item
The restriction of the mapping $\UMeas(\iota)$ to the positive cone
$\UMeas(X)^+$ with its $\ueb$ topology is a topological embedding
into $\UMeas(Y)^+$ with its $\ueb$ topology.
\item
If $\UMeas(\iota)$ maps $\mmeas\iin\UMeas(X)$ to $\nmeas\iin\UMeas(Y)$ then
it maps $\mmeas^+$ to $\nmeas^+$ and $\mmeas^-$ to $\nmeas^-$.
\end{enumerate}
\end{theorem}

\begin{proof}
The image of $\mmeas$ under $\UMeas(\iota)$
is the restriction of $\mmeas$ to the space $\Ub(Y)\subseteq\Ub(X)$.

1. If $\mmeas\iin\UMeas(X)$ is such that $\mmeas(g)=0$
for every $g\iin\Ub(Y)$
then $\mmeas(f)=0$ for every $f\iin\Ub(X)$
by Lemma~\ref{lem:suffcond}.
That means that the kernel of the linear mapping $\UMeas(\iota)$ is $\{0\}$.

2. Take any $\mmeas\iin\UMeas(X)^+$
and a net $\{\mmeas_\gamma \}_\gamma$ in $\UMeas(X)^+$
such that $\lim_\gamma \mmeas_\gamma(g) = \mmeas(g)$
for every $g\iin\Ub(Y)$.
Take any $f\iin\Ub(X)$ and $\varepsilon>0$.

By Lemma~\ref{lem:suffcond} there are $g_0,g_1\iin\Ub(Y)$
such that $g_0 \leq f \leq g_1$ and $\mmeas(g_1 - g_0)<\varepsilon$.
For almost all $\gamma$ we have
$\mmeas (g_0) - \mmeas_\gamma(g_0) < \varepsilon$
and
$\mmeas_\gamma (g_1) - \mmeas(g_1) < \varepsilon$,
therefore
\begin{align*}
\mmeas(f) - \mmeas_\gamma(f)
& \leq \mmeas(f-g_0) + \mmeas(g_0) - \mmeas_\gamma(g_0) < 2\varepsilon \\
\mmeas_\gamma(f) - \mmeas(f)
& \leq \mmeas_\gamma(g_1) - \mmeas(g_1) + \mmeas(g_1 - f)  < 2\varepsilon
\end{align*}
Hence $\lim_\gamma \mmeas_\gamma(f) = \mmeas(f)$.

This proves that the mapping $\UMeas(\iota)$
is a topological embedding of $\UMeas(X)^+$ with the $\Ub(X)$-weak
topology into $\UMeas(Y)^+$ with the $\Ub(Y)$-weak topology.
By~\cite[6.13]{Pachl2013usm}
the $\Ub(\cdot)$-weak topology and the $\ueb$ topology coincide on
the positive cones $\UMeas(X)^+$ and $\UMeas(Y)^+$,
and therefore $\UMeas(\iota)$ is also a topological embedding when
$\UMeas(X)^+$ and $\UMeas(Y)^+$ have their $\ueb$ topologies.

3. Let $\nmeas=\UMeas(\iota)(\mmeas)$,
so that $\nmeas(h)=\mmeas(h)$ for all $h\iin\Ub(Y)$.
Let $\nmeas^\prime=\UMeas(\iota)(\mmeas^+)$,
and take any $h\iin\Ub(Y)^+$ and $\varepsilon>0$.
By definition,
\begin{align*}
\nmeas^\prime(h) & = \sup \{ \mmeas(f) \mid f\iin\Ub(X)^+, f \leq h \} \\
\nmeas^+(h) & = \sup \{ \mmeas(g) \mid g\iin\Ub(Y)^+, g \leq h \}
\end{align*}
By Lemma~\ref{lem:suffcond},
for every $f\iin\Ub(X)^+$ there is $g\iin\Ub(Y)^+$ such that $g \leq f$
and $\abs{\mmeas(f)-\mmeas(g)}<\varepsilon$.
Hence $\abs{\nmeas^\prime(h)-\nmeas^+(h)}\leq\varepsilon$.

It follows that $\nmeas^+=\UMeas(\iota)(\mmeas^+)$ and
$\nmeas^- = \nmeas^+ - \nmeas = \UMeas(\iota)(\mmeas^+ - \mmeas)
= \UMeas(\iota)(\mmeas^-)$.
\end{proof}

The forthcoming general embedding theorem
deals with four uniform spaces on the same set $S$:
$X_0$, $X_1$, $X_\vee :=X_0 \vee X_1$ and $X_\wedge := X_0 \wedge X_1$.
We consider the identity mapping $\iota\colon S\to S$
in five different roles as a uniformly continuous mapping,
along with its unique extensions to continuous linear mappings
between the spaces $\UMeas(\cdot)$.
In the diagram and in what follows
I write $\extm{\iota_{\alpha\beta}}$ instead of $\UMeas(\iota_{\alpha\beta})$,
in the interest of readability.

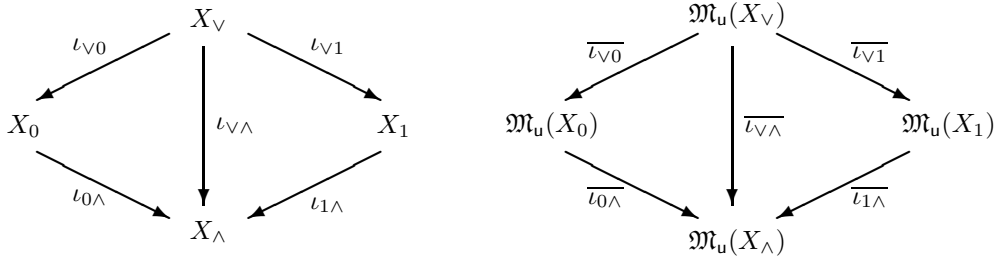
\begin{picture}(200,120)(-80,20)
\thicklines

\put(-10,110){\makebox(30,20){${X_\vee}$}}
\put(-80,70){\makebox(30,20){${X_0}$}}
\put(60,70){\makebox(30,20){${X_1}$}}
\put(-10,30){\makebox(30,20){${X_\wedge}$}}

\put(-10,115){\vector(-2,-1){50}}
\put(20,115){\vector(2,-1){50}}
\put(70,70){\vector(-2,-1){50}}
\put(-60,70){\vector(2,-1){50}}
\put(3,110){\vector(0,-1){60}}

\put(-55,100){\makebox(30,20){$\iota_{\vee 0}$}}
\put(35,100){\makebox(30,20){$\iota_{\vee 1}$}}
\put(-55,43){\makebox(30,20){$\iota_{0\wedge}$}}
\put(35,40){\makebox(30,20){$\iota_{1\wedge}$}}
\put(0,70){\makebox(30,20){$\initfin$}}

\put(190,112){\makebox(30,20){$\UMeas(X_\vee)$}}
\put(120,70){\makebox(30,20){$\UMeas(X_0)$}}
\put(270,70){\makebox(30,20){$\UMeas(X_1)$}}
\put(190,26){\makebox(30,20){$\UMeas(X_\wedge)$}}

\put(190,115){\vector(-2,-1){50}}
\put(220,115){\vector(2,-1){50}}
\put(270,70){\vector(-2,-1){50}}
\put(140,70){\vector(2,-1){50}}
\put(203,110){\vector(0,-1){60}}

\put(140,98){\makebox(30,20){$\extm{\iota_{\vee 0}}$}}
\put(240,98){\makebox(30,20){$\extm{\iota_{\vee 1}}$}}
\put(140,43){\makebox(30,20){$\extm{\iota_{0\wedge}}$}}
\put(240,43){\makebox(30,20){$\extm{\iota_{1\wedge}}$}}
\put(200,70){\makebox(30,20){$\extm{\initfin}$}}
\end{picture}

\noindent
To simplify the notation,
I let $\alpha\beta$ stand for any of
$\vee 0$, $\vee 1$, $0\wedge$, $1\wedge$, $\vee\wedge$,
and treat it as an ordered pair of indices;
for example, if $\alpha\beta$ is $\vee 0$ then $\alpha$ is $\vee$,
$\beta$ is $0$, $X_{\alpha}$ is $X_\vee$
and $\iota_{\alpha\beta}$ is $\iota_{\vee 0}$.

\begin{theorem}
    \label{th:embedgen}
Let $\pset_0$ and $\pset_1$ be two point-separating sets of
pseudometrics on a set $S$.
For $j=0,1$ denote by $X_j$ the uniform space induced by $\pset_j$,
and write $X_\vee :=X_0 \vee X_1$, $X_\wedge := X_0 \wedge X_1$.
For $\alpha\beta \iin \{\vee 0, \vee 1, 0\wedge, 1\wedge, \vee\wedge\}$,
let $\iota_{\alpha\beta}\colon X_\alpha \to X_\beta$
be the uniformly continuous identity mapping.
Assume that
\begin{itemize}
\item
$\psmfun{z}{\psm}\iin\Uall(X_1)$ for all $\psm\iin\pset_0$ and $z\iin S$;
\item
$\psmfun{z}{\psm}\iin\Uall(X_0)$ for all $\psm\iin\pset_1$ and $z\iin S$.
\end{itemize}

Then for every
$\alpha\beta \iin \{\vee 0, \vee 1, 0\wedge, 1\wedge, \vee\wedge\}$
the mapping
$\extm{\iota_{\alpha\beta}}=\UMeas(\iota_{\alpha\beta})$ from
$\UMeas(X_\alpha)$ to $\UMeas(X_\beta)$ is injective
and its restriction to the positive cone $\UMeas(X_\alpha)^+$
with its $\ueb$ topology
is a topological embedding into $\UMeas(X_\beta)^+$
with its $\ueb$ topology.
Moreover,
\begin{equation}
    \label{eq:coinspec}
\extm{\initfin}(\UMeas(X_\vee))
=\extm{\iota_{0\wedge}}(\UMeas(X_0))
\,\cap\, \extm{\iota_{1\wedge}}(\UMeas(X_1)).
\tag{$\bumpeq$}
\end{equation}
\end{theorem}

\begin{proof}
To prove the first statement for $\extm{\iota_{\vee 0}}$,
apply Theorem~\ref{th:prep} with $X=X_\vee$, $Y=X_0$
and $\pset = \pset_0 \cup \pset_1$.
Take any $z\iin S$.
If $\psm\iin\pset_0$ then $\psmfun{z}{\psm}\iin\Uall(X_0)$
by the definition of $X_0$.
If $\psm\iin\pset_1$ then $\psmfun{z}{\psm}\iin\Uall(X_0)$ by the assumption.
Hence $\extm{\iota_{\vee 0}}$ is injective and its restriction to
$\UMeas(X_\vee)^+$ is a topological embedding.
The same holds for $\extm{\iota_{\vee 1}}$ by
Theorem~\ref{th:prep} with $X=X_\vee$, $Y=X_1$.

To prove the first statement for $\extm{\iota_{0\wedge}}$,
apply Theorem~\ref{th:prep} with $X=X_0$, $Y=X_\wedge$ and $\pset=\pset_0$.
For $z\iin S$ and $\psm\iin\pset$ we have
$\psmfun{z}{\psm}\iin\Uall(X_0)$ by the definition of $X_0$,
and $\psmfun{z}{\psm}\iin\Uall(X_1)$ by the assumption,
hence $\psmfun{z}{\psm}\iin\Uall(X_0)\cap\Uall(X_1)=\Uall(X_\wedge)$.
Thus $\extm{\iota_{0\wedge}}$ is injective and its restriction to
$\UMeas(X_0)^+$ is a topological embedding.
The same holds for $\extm{\iota_{1\wedge}}$
by Theorem~\ref{th:prep} with $X=X_1$, $Y=X_\wedge$ and $\pset=\pset_1$.

The mapping $\extm{\initfin}$ is injective and its restriction
to $\UMeas(X_\vee)^+$ is a topological embedding
because
$\extm{\initfin}= \extm{\iota_{0\wedge}} \circ \extm{\iota_{\vee 0}}$.

To prove the equality \thetag{\ref{eq:coinspec}},
take any
$\mmeas\iin\extm{\iota_{0\wedge}}(\UMeas(X_0))
\cap\extm{\iota_{1\wedge}}(\UMeas(X_1))$.
There are $\mmeas_0 \iin\UMeas(X_0)$
and $\mmeas_1 \iin\UMeas(X_1)$
for which
$\mmeas=\extm{\iota_{0\wedge}}(\mmeas_0)=\extm{\iota_{1\wedge}}(\mmeas_1)$.
First I prove that
$\mmeas\iin\extm{\initfin}(\UMeas(X_\vee))$ under the additional
assumption that $\mmeas\geq 0$.
By~\cite[6.6]{Pachl2013usm} the set
$\Mol(X_\wedge)^+$ is dense in $\UMeas(X_\wedge)^+$.
Since obviously $\Mol(X_\wedge)^+ = \extm{\initfin} (\Mol(X_\vee)^+)$,
there is a net $\{\mmeas_\gamma\}_\gamma$ in $\Mol(X_\vee)^+$ such that
\[
\lim_\gamma \extm{\initfin}(\mmeas_\gamma)=\mmeas
\quad\text{in } \;\UMeas(X_\wedge)^+.
\]
By the already proved statements about topological embeddings,
\begin{align*}
\lim_\gamma \extm{\iota_{\vee 0}}(\mmeas_\gamma)=\mmeas_0
\quad\text{in } \;\UMeas(X_0)^+   \\
\lim_\gamma \extm{\iota_{\vee 1}}(\mmeas_\gamma)=\mmeas_1
\quad\text{in } \;\UMeas(X_1)^+
\end{align*}
Assume, without loss of generality, that there is $r>0$ such that
$\norm{\mmeas_\gamma}\leq r$ for all $\gamma$.
The set
\[
B := \{ \nmeas \iin \Ub(X_\vee)^\ast
\mid \nmeas \geq 0 \text{ and } \norm{\nmeas} \leq r \}
\]
is $\Ub(X_\vee)$-weakly compact and $\mmeas_\gamma \iin B$ for all $\gamma$.
Hence the net $\{\mmeas_\gamma\}_\gamma$ has a cluster point $\nmeas\iin B$
in the $\Ub(X_\vee)$-weak topology.

Now $\nmeas$ is mapped to $\mmeas_0$ by the natural extension
of $\extm{\iota_{\vee 0}}$ to a $\Ub$-weakly continuous linear mapping
$\Ub(X_\vee)^\ast \to \Ub(X_0)^\ast$,
and it is mapped to $\mmeas_1$ by the extension
of $\extm{\iota_{\vee 1}}$ to a $\Ub$-weakly continuous linear mapping
$\Ub(X_\vee)^\ast \to \Ub(X_1)^\ast$.
By~\cite[6.25]{Pachl2013usm} we get $\nmeas\iin\UMeas(X_\vee)$.
Since $\mmeas_0 = \extm{\iota_{\vee 0}} (\nmeas)$ and
$\mmeas_1 = \extm{\iota_{\vee 1}} (\nmeas)$,
it follows that $\mmeas =\extm{\initfin}(\nmeas)$.

Finally I prove $\mmeas\iin\extm{\initfin}(\UMeas(X_\vee))$
without the assumption $\mmeas\geq 0$.
By part~3 of Theorem~\ref{th:prep} we have
$\mmeas^+=\extm{\iota_{0\wedge}}(\mmeas_0^+)
=\extm{\iota_{1\wedge}}(\mmeas_1^+)$
and
$\mmeas^-=\extm{\iota_{0\wedge}}(\mmeas_0^-)
=\extm{\iota_{1\wedge}}(\mmeas_1^-)$.
Hence there are $\nmeas_0,\nmeas_1\iin\UMeas(X_\vee)$
for which $\extm{\initfin}(\nmeas_0)=\mmeas^+$ and
$\extm{\initfin}(\nmeas_1)=\mmeas^-$.
Then $\extm{\initfin}(\nmeas_0-\nmeas_1)=\mmeas$.
\end{proof}


\section{Embedding theorem for topological groups}
    \label{sec:embgroup}

In this section I apply Theorem~\ref{th:embedgen} to the four
uniform structures on a topological group.
Although Theorem~\ref{th:embedgroup} is a special case of
Theorem~\ref{th:embedquot} in the next section,
I present the former first because it captures the essence
of the key result without additional complications.

When $\psm$ is a left invariant continuous pseudometric
on a topological group $G$ and $z\iin G$,
the function $\psmfun{z}{\psm}$ is \emph{left} uniformly continuous.
Less obviously, $\psmfun{z}{\psm}$ is also \emph{right} uniformly continuous:

\begin{lemma}
    \label{lem:spikef}
Let $G$ be a topological group,
$\psm$ a left invariant continuous pseudometric on $G$,
and $z\iin G$.
Then  $\psmfun{z}{\psm}\iin\Uall(\rightu G)$.
\end{lemma}

\begin{proof}
Define the pseudometric $\psm^\prime$ by
$\psm^\prime (x,y) := \psm(x^{-1} z , y^{-1} z)$, $x,y\iin G$.
Then $\psm^\prime$ is continuous and right invariant,
and for $x,y\iin G$ we have
\[
\abs{ \psmfun{z}{\psm}(x) -  \psmfun{z}{\psm}(y)}
= \abs{ \psm( e_G, x^{-1} z) - \psm( e_G, y^{-1} z) }
\leq \psm(x^{-1} z, y^{-1} z)
= \psm^\prime (x,y) .   \qedhere
\]
\end{proof}

The following embedding theorem
deals with natural mappings between the spaces $\UMeas(\cdot)$
for the four uniformities on a topological group.

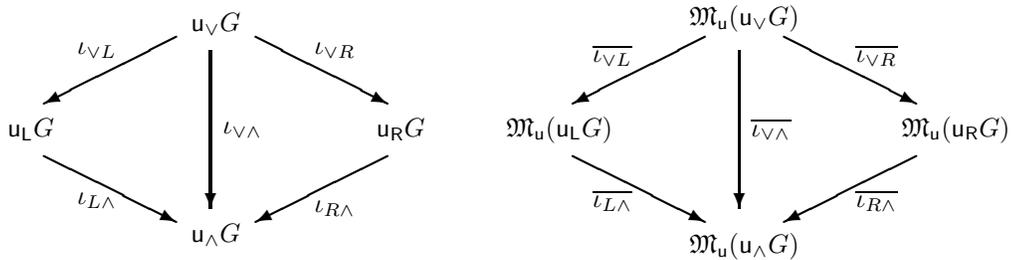
\begin{picture}(200,120)(-80,20)
\thicklines

\put(-10,110){\makebox(30,20){${\upu G}$}}
\put(-80,70){\makebox(30,20){${\leftu G}$}}
\put(60,70){\makebox(30,20){${\rightu G}$}}
\put(-10,30){\makebox(30,20){${\lowu G}$}}

\put(-10,115){\vector(-2,-1){50}}
\put(20,115){\vector(2,-1){50}}
\put(70,70){\vector(-2,-1){50}}
\put(-60,70){\vector(2,-1){50}}
\put(3,110){\vector(0,-1){60}}

\put(-55,100){\makebox(30,20){$\initL$}}
\put(35,100){\makebox(30,20){$\initR$}}
\put(-55,43){\makebox(30,20){$\finL$}}
\put(35,40){\makebox(30,20){$\finR$}}
\put(0,70){\makebox(30,20){$\initfin$}}

\put(190,112){\makebox(30,20){$\UMeas(\upu G)$}}
\put(120,70){\makebox(30,20){$\UMeas(\leftu G)$}}
\put(270,70){\makebox(30,20){$\UMeas(\rightu G)$}}
\put(190,26){\makebox(30,20){$\UMeas(\lowu G)$}}

\put(190,115){\vector(-2,-1){50}}
\put(220,115){\vector(2,-1){50}}
\put(270,70){\vector(-2,-1){50}}
\put(140,70){\vector(2,-1){50}}
\put(203,110){\vector(0,-1){60}}

\put(140,98){\makebox(30,20){$\extm{\initL}$}}
\put(240,98){\makebox(30,20){$\extm{\initR}$}}
\put(140,43){\makebox(30,20){$\extm{\finL}$}}
\put(240,43){\makebox(30,20){$\extm{\finR}$}}
\put(200,70){\makebox(30,20){$\extm{\initfin}$}}
\end{picture}

\begin{theorem}
    \label{th:embedgroup}
Let $G$ be a topological group.
For $\alpha\beta \iin \{\vee L, \vee R, L\wedge, R\wedge, \vee\wedge\}$
let $\iota_{\alpha\beta}\colon\genu{\alpha} G \to \genu{\beta} G$
be the uniformly continuous identity mapping.

Then for every
$\alpha\beta \iin \{\vee L, \vee R, L\wedge, R\wedge, \vee\wedge\}$
the mapping
$\extm{\iota_{\alpha\beta}}=\UMeas(\iota_{\alpha\beta})$ from
$\UMeas(\genu{\alpha} G)$ to $\UMeas(\genu{\beta} G)$ is injective
and its restriction to the positive cone $\UMeas(\genu{\alpha} G)^+$
with its $\ueb$ topology
is a topological embedding into $\UMeas(\genu{\beta} G)^+$
with its $\ueb$ topology.
Moreover,
\begin{equation*}
    \label{eq:coing}
\extm{\initfin}(\UMeas(\upu G))
= \extm{\finL}(\UMeas(\leftu G)) \,\cap\, \extm{\finR}(\UMeas(\rightu G)) .
\end{equation*}
\end{theorem}

\begin{proof}
Apply Theorem~\ref{th:embedgen}
with the set $\pset_0$ of left invariant continuous pseudometrics,
and the set $\pset_1$ of right invariant continuous pseudometrics on $G$.
The assumption in the theorem holds for $\pset_0$ by Lemma~\ref{lem:spikef}
and for $\pset_1$ by left--right symmetry.
\end{proof}


\section{Embedding theorem for quotient spaces}
    \label{sec:embquot}

In this section I apply Theorem~\ref{th:embedgen} to four
uniform structures on the quotient space of a topological group $G$
by a closed subgroup ${H}$.
Under the assumption that ${H}$ is neutral,
this yields an embedding theorem that parallels Theorem~\ref{th:embedgroup}.
It will be seen that the condition of ${H}$ being neutral is also necessary.

The first step, an analogue of Lemma~\ref{lem:spikef},
now has two forms, left and right,
because in $G/{H}$ we no longer have the perfect left--right symmetry.
The simpler of the two is the left version in the next lemma.

\begin{lemma}
    \label{lem:spikefleft}
Let $G$ be a topological group, ${H}$ its closed subgroup,
$\psm$ a left invariant continuous pseudometric on $G/{H}$
and $A=z{H}\iin G/{H}$.
Then $\psmfun{A}{\psm}\iin\Uall(\rightu G/{H})$.
\end{lemma}

\begin{proof}
Let $q\colon G \to G/{H}$ be the quotient mapping.
Define the pseudometric $\psm^\prime$ on $G$ by
\[
\psm^\prime (x,y) := \psm(x^{-1} z{H} , y^{-1} z{H}), \quad x,y\iin G .
\]
Then $\psm^\prime$ is continuous and right invariant,
and for $x,y\iin G$ we have
\[
\abs{ \psmfun{A}{\psm}(x{H}) -  \psmfun{A}{\psm}(y{H})}
= \abs{ \psm( {H}, x^{-1} z{H}) - \psm( {H}, y^{-1} z{H}) }
\leq \psm( x^{-1} z{H}, y^{-1} z{H})
= \psm^\prime (x,y) .
\]
Hence the function $\psmfun{A}{\psm} \circ q$ is in $\Uall(\rightu G)$,
which means that $\psmfun{A}{\psm}\iin\Uall(\rightu G/{H})$.
\end{proof}

The second analogue of Lemma~\ref{lem:spikef} relies on
a modified triangle inequality for $\psmtil$.

\begin{lemma}
    \label{lem:estima}
Let $G$ be a topological group, ${H}$ its closed subgroup,
$\psm$ a bounded right invariant continuous pseudometric on $G$,
and $x,y,z\iin G$.
Then
\[
\abslarge{\psmtil(z{H},x{H}) - \psmtil(z{H},y{H})}
\leq \psmtil (z{H}x^{-1},z{H}y^{-1}) .
\]
\end{lemma}

\begin{proof}
Since $\psm$ is right invariant,
$\psmtil(A,B)=\psm(A,B)$ for any two cosets $A,B\iin G/{H}$.
Hence
\begin{align*}
\psmtil(z{H},x{H}) - \psmtil(z{H},y{H})
& =\inf_{s\in {H}} \psm(zsx^{-1},e_G) -\inf_{t\in {H}} \psm(zty^{-1},e_G)\\
& =\sup_{t\in {H}} \inf_{s\in {H}}
                  \left(\psm(zsx^{-1},e_G) - \psm(zty^{-1},e_G) \right) \\
& \leq \sup_{t\in {H}} \inf_{s\in {H}} \;\psm(zsx^{-1}, zty^{-1})       \\
& = \sup_{t\in {H}} \;\psm(z{H}x^{-1}, zty^{-1}) .
\end{align*}
By the same argument with $x$ and $y$ exchanged,
\[
\psmtil(z{H},y{H}) - \psmtil(z{H},x{H})
\leq \sup_{t\in {H}} \;\psm(ztx^{-1}, z{H}y^{-1}).
\qedhere
\]
\end{proof}

The next proposition characterizes neutral subgroups
in terms of pseudometrics.

\begin{proposition}
    \label{prop:neutralequiv}
Let $G$ be a topological group.
The following properties of a closed subgroup ${H}$ of $G$ are equivalent:
\begin{enumerate}
\item[(i)]
${H}$ is neutral.
\item[(ii)]
The topology of the uniform space $\leftu G/{H}$ is the quotient topology
of $G/{H}$.
\item[(iii)]
For every bounded right invariant continuous pseudometric $\psm$ on $G$
and every $z\iin G$,
the pseudometric $\psm^\prime$ defined by
$\psm^\prime(x,y):=\psmtil(z{H}x^{-1},z{H}y^{-1})$, $x,y\iin G$,
is continuous on $G$.
\item[(iv)]
For every bounded right invariant continuous pseudometric $\psm$ on $G$,
the pseudometric $\psm^\prime$ defined by
$\psm^\prime(x,y):=\psmtil({H}x^{-1},{H}y^{-1})$, $x,y\iin G$,
is continuous on $G$.
\item[(v)]
For every bounded right invariant continuous pseudometric $\psm$ on $G$
and every $A\iin G/{H}$,
the function $\psmfun{A}{\psmtil}$ is in $\Uall(\leftu G/{H})$.
\item[(vi)]
For every bounded right invariant continuous pseudometric $\psm$ on $G$,
the function $\psmfun{{H}}{\psmtil}$ is in $\Uall(\leftu G/{H})$.
\end{enumerate}
\end{proposition}

\begin{proof}
(i)$\Leftrightarrow$(ii) by~\cite[5.28]{Roelcke1981ust}.

To prove (i)$\Rightarrow$(iii), assume (i) and let $\psm$ be
any bounded right invariant continuous pseudometric on $G$.
Let $\psm^\prime(x,y):=\psmtil(z{H}x^{-1},z{H}y^{-1})$, $x,y\iin G$,
and take any $\varepsilon>0$.
As $\psm$ is right invariant,
\[
\psm^\prime(x,e_G)
= \max\left( \;\sup_{t\in {H}} \;\psm(z{H},ztx), \;
             \sup_{t\in {H}} \;\psm(z{H},ztx^{-1}) \;
      \right)
\]

The set $U:=\{y\iin G \mid \psm(z{H},zy) < \varepsilon \}$ is
a neighbourhood of $e_G$ and $U{H}=U$.
As ${H}$ is neutral, there is a neighbourhood $V$ of $e_G$
such that ${H}V \subseteq U$.
I claim that $\psm^\prime(x,e_G)\leq\varepsilon$
for every $x\iin V\cap V^{-1}$.

If $x\iin V$ and $t\iin {H}$ then $tx \iin U$,
hence $\psm(z{H},ztx) < \varepsilon$.
It follows that
$
\sup_{t\in {H}} \psm(z{H},ztx)  \leq \varepsilon
$.
If $x^{-1}\iin V$ then
$\sup_{t\in {H}} \psm(z{H},ztx^{-1})  \leq \varepsilon$
by the same argument.
That proves the claim.
Thus the function $x\mapsto \psm^\prime(x,e_G)$ is continuous.
Since the pseudometric $\psm^\prime$ is left invariant,
it follows that it is continuous.

Obviously (iii)$\Rightarrow$(iv) and (v)$\Rightarrow$(vi).
To prove (iii)$\Rightarrow$(v) and (iv)$\Rightarrow$(vi),
note that the pseudometric $\psm^\prime$ in (iii) or (iv)
is left invariant and apply Lemma~\ref{lem:estima}.

To prove (vi)$\Rightarrow$(i),
assume (vi) and take any neighbourhood $U$ of $e_G$.
There is a bounded right invariant continuous pseudometric $\psm$
such that $\{x\iin G \mid \psm(e_G,x) < 1 \}\subseteq U$.
By the assumption there is a left invariant continuous pseudometric
$\psm^\prime$ on $G/{H}$ such that if $\psm^\prime({H},x{H})<1$
then $\psmtil({H},x{H})<1$.
The set $V:=\{x\iin G \mid \psm^\prime({H},x{H}) < 1 \}$
is a neighbourhood of $e_G$ in $G$ and ${H} V = V$.
Since $\psmtil({H},x{H})=\psm({H}, x)$, we get
\[
{H} V = V
= \{x\iin G \mid \psm^\prime({H},x{H}) < 1 \}
\subseteq \{x\iin G \mid \psm({H},x) < 1 \}
\subseteq U{H} .
\qedhere
\]
\end{proof}

We are ready for the embedding theorem for quotient spaces.
In the theorem, ${H}$ is a closed neutral subgroup of a topological group $G$.
The identity mapping $\iota\colon G/{H} \to G/{H}$ again plays five
different roles as a uniformly continuous mapping.
In the following diagram I write
$Q_L=\leftu G/{H}$,
$Q_R=\rightu G/{H}$,
$Q_\vee=(\leftu G/{H})\vee (\rightu G/{H})$,
and $Q_\wedge=(\leftu G/{H})\wedge (\rightu G/{H})$.

From the definition it immediately follows that
$\upu G/{H}$ is finer than $Q_\vee$
and $Q_\wedge$ is finer than $\lowu G/{H}$.
In fact, $Q_\wedge=\lowu G/{H}$~\cite[5.31(c)]{Roelcke1981ust}.
On the other hand, $\upu G/{H}$ may be strictly finer
than $Q_\vee$~\cite[6.22]{Roelcke1981ust}.

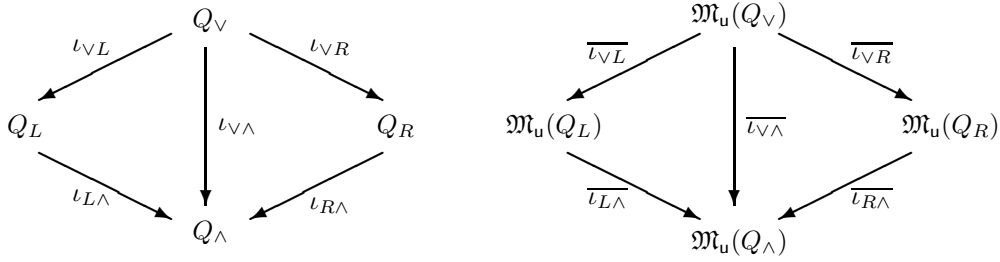
\begin{picture}(200,120)(-80,20)
\thicklines

\put(-10,110){\makebox(30,20){${Q_\vee}$}}
\put(-80,70){\makebox(30,20){${Q_L}$}}
\put(60,70){\makebox(30,20){${Q_R}$}}
\put(-10,30){\makebox(30,20){${Q_\wedge}$}}

\put(-10,115){\vector(-2,-1){50}}
\put(20,115){\vector(2,-1){50}}
\put(70,70){\vector(-2,-1){50}}
\put(-60,70){\vector(2,-1){50}}
\put(3,110){\vector(0,-1){60}}

\put(-55,100){\makebox(30,20){$\initL$}}
\put(35,100){\makebox(30,20){$\initR$}}
\put(-55,43){\makebox(30,20){$\finL$}}
\put(35,40){\makebox(30,20){$\finR$}}
\put(0,70){\makebox(30,20){$\initfin$}}

\put(190,112){\makebox(30,20){$\UMeas(Q_\vee)$}}
\put(120,70){\makebox(30,20){$\UMeas(Q_L)$}}
\put(270,70){\makebox(30,20){$\UMeas(Q_R)$}}
\put(190,26){\makebox(30,20){$\UMeas(Q_\wedge)$}}

\put(190,115){\vector(-2,-1){50}}
\put(220,115){\vector(2,-1){50}}
\put(270,70){\vector(-2,-1){50}}
\put(140,70){\vector(2,-1){50}}
\put(203,110){\vector(0,-1){60}}

\put(140,98){\makebox(30,20){$\extm{\initL}$}}
\put(240,98){\makebox(30,20){$\extm{\initR}$}}
\put(140,43){\makebox(30,20){$\extm{\finL}$}}
\put(240,43){\makebox(30,20){$\extm{\finR}$}}
\put(200,70){\makebox(30,20){$\extm{\initfin}$}}
\end{picture}

\begin{theorem}
    \label{th:embedquot}
Let $G$ be a topological group,
and ${H}$ a closed neutral subgroup of $G$.
Write
$Q_L=\leftu G/{H}$,
$Q_R=\rightu G/{H}$,
$Q_\vee=(\leftu G/{H})\vee (\rightu G/{H})$,
and $Q_\wedge=(\leftu G/{H})\wedge (\rightu G/{H})$.
For $\alpha\beta \iin \{\vee L, \vee R,
L\wedge, R\wedge, \vee\wedge\}$
let $\iota_{\alpha\beta}\colon Q_\alpha \to Q_\beta$
be the uniformly continuous identity mapping.

Then for every
$\alpha\beta \iin \{\vee L,\vee R,
L\wedge,R\wedge,\vee\wedge\}$
the mapping
$\extm{\iota_{\alpha\beta}}=\UMeas(\iota_{\alpha\beta})$ from
$\UMeas(Q_\alpha)$ to $\UMeas(Q_\beta)$ is injective
and its restriction to the positive cone $\UMeas(Q_\alpha)^+$
with its $\ueb$ topology
is a topological embedding into $\UMeas(Q_\beta)^+$
with its $\ueb$ topology.
Moreover,
\begin{equation*}
    \label{eq:coinquot}
\extm{\initfin}(\UMeas(Q_\vee))
= \extm{\finL}(\UMeas(Q_L))
\,\cap\, \extm{\finR}(\UMeas(Q_R)) .
\end{equation*}
\end{theorem}

\begin{proof}
To apply Theorem~\ref{th:embedgen},
let $\pset_0$ be the set of left invariant continuous pseudometrics
on $G/{H}$.
Let $\pset_1$ be the set of the pseudometrics $\psmtil$ on $G/{H}$
where $\psm$ is a bounded right invariant continuous pseudometric on $G$.

By~\cite[7.7]{Roelcke1981ust},
$\pset_0$ induces $Q_L=\leftu G/{H}$.
By~\cite[7.6]{Roelcke1981ust},
$\pset_1$ induces $Q_R=\rightu G/{H}$.
The assumption in Theorem~\ref{th:embedgen} holds
by Lemma~\ref{lem:spikefleft} and Proposition~\ref{prop:neutralequiv}.
\end{proof}

If the natural mapping from $\UMeas(Q_\vee)$ to $\UMeas(Q_L)$
is a topological embedding then in particular the topology
of $Q_L=\leftu G/{H}$ is the same as the topology of $Q_\vee$
which is the quotient topology of $G/{H}$.
Thus, by (i)$\Leftrightarrow$(ii) in~\ref{prop:neutralequiv},
${H}$ being a neutral subgroup of $G$ is not only sufficient
but also necessary for the conclusion of Theorem~\ref{th:embedquot}.
For the same reason it is also a necessary condition for the conclusion
of the Roelcke--Dierolf embedding theorem~\cite[10.9]{Roelcke1981ust}.


\section{Counterexample}
    \label{sec:freeunif}

For a uniform space $X$,
an ``unbounded version'' of $\UMeas(X)$ is the space $\FMeas(X)$ of
\emph{free uniform measures}~\cite[Ch.10]{Pachl2013usm}.
The topology of $\FMeas(X)$ is that of uniform convergence on
uniformly equicontinuous pointwise bounded sets of functions,
called the \emph{$\ue$ topology}.
As a linear space,
$\FMeas(X)$ naturally identifies with a subspace of $\UMeas(X)$,
but the $\ue$ topology on $\FMeas(X)$ is in general finer
that the $\ueb$ topology that $\FMeas(X)$ inherits from $\UMeas(X)$.
The two topologies agree on $\FMeas(X)$ if and only if
$\Uall(X)=\Ub(X)$, and in that case $\FMeas(X)=\UMeas(X)$.

The following example shows that,
in spite of similarities between the spaces $\UMeas(X)$ and $\FMeas(X)$,
the embedding theorems proved here for $\UMeas$ do not hold for $\FMeas$.
When the spaces $\UMeas(\cdot)$ are replaced by $\FMeas(\cdot)$
in Theorems~\ref{th:embedgroup} and~\ref{th:embedquot},
the mappings $\FMeas(\iota_{\alpha\beta})$ need not be topological embeddings.

Let $G$ be a topological group such that there is an unbounded
function $f\iin\Uall(\rightu G)$ but all functions
in $\Uall(\lowu G)$ are bounded.
For example, we may take $G$ to be the infinite symmetric group
with the topology of pointwise convergence,
which satisfies $\Uall(\lowu G)=\Ub(\lowu G)$
because $\lowu G$ is precompact~\cite[9.14]{Roelcke1981ust}.

Take points $x_j \iin G$ such that $f(x_j)\geq j$ for $j=1,2,\dots$.
Define $\mmeas_j := \frac{1}{f(x_j)} \, \pmass(x_j)$, $j=1,2,\dots$.
Then $\lim_j \mmeas_j = 0$ in the $\ue$ topology on $\FMeas(\lowu G)^+$.
On the other hand, $\mmeas_j$ do not converge to $0$ in the $\ue$ topology
on $\FMeas(\rightu G)^+$ because $\mmeas_j(f)=1$ for every $j$.
Thus the natural mapping
$\FMeas(\rightu G)^+ \to \FMeas(\lowu G)^+$ is not a topological embedding.


\end{document}